\documentclass[11pt,a4paper]{article}

\usepackage{amsmath,amssymb,amsthm,mathtools}
\usepackage{geometry}
\usepackage{enumitem}
\usepackage{microtype}
\usepackage{xcolor}
\usepackage{hyperref}

\hypersetup{
	colorlinks=true,
	linkcolor=blue!55!black,
	citecolor=blue!55!black,
	urlcolor=blue!55!black,
	pdftitle={Involution and Commutator Length in PU(n,1)},
	pdfauthor={Zhongqi Wang and Shihai Yang},
	pdfkeywords={Complex hyperbolic space, Involution length, Holomorphic involution},
}
\setlist{nosep,leftmargin=2em}
\allowdisplaybreaks[3]
\theoremstyle{plain}
\newtheorem{theorem}{Theorem}[section]
\newtheorem{proposition}[theorem]{Proposition}
\newtheorem{lemma}[theorem]{Lemma}
\newtheorem{corollary}[theorem]{Corollary}
\theoremstyle{definition}
\newtheorem{definition}[theorem]{Definition}
\theoremstyle{remark}

\newcommand{\C}{\mathbb C}
\newcommand{\R}{\mathbb R}
\newcommand{\T}{\mathbb T}
\newcommand{\CP}{\mathbb{CP}}
\newcommand{\HC}{\mathbf H_{\C}^{n}}
\newcommand{\Id}{\mathrm I}
\newcommand{\PU}{\mathrm{PU}}

\newcommand{\UU}{\mathrm U}

\newcommand{\spec}{\operatorname{spec}}
\newcommand{\tr}{\operatorname{tr}}
\newcommand{\im}{\operatorname{im}}
\newcommand{\kerx}{\operatorname{ker}}
\newcommand{\diag}{\operatorname{diag}}
\newcommand{\linv}{\ell_{\mathcal I}}
\newcommand{\rootlinv}{\operatorname{srl}_{\mathcal I}}

\newcommand{\ip}[2]{\left\langle #1,#2\right\rangle}
\newcommand{\orth}{\mathbin{\perp}}

\title{Involution and Commutator Length in $\PU(n,1)$}
\author{
	Zhongqi Wang \qquad Shihai Yang\thanks{Corresponding author.}\\[0.5em]
	\small School of Mathematics, Shanghai University of Finance and Economics\\
	\small 777 Guoding Road, Yangpu District, Shanghai 200433, P. R. China\\
	\small \texttt{wangzhongqi2021@126.com} \qquad
	\small \texttt{yang.shihai@mail.shufe.edu.cn}
}
\date{}

\begin{document}
	\maketitle
	
	\begin{abstract}
		We study decompositions of holomorphic isometries of complex hyperbolic
		space into holomorphic involutions. We prove that, for every $n\ge3$, the
		involution length of $\PU(n,1)$ is $4$. This improves the higher-dimensional
		upper bound $8$ of Paupert--Will, as well as the projective upper bound $5$
		implied by B\"unger's linear decomposition theorem, to the optimal value
		$4$. Combined with the two-dimensional result of Paupert--Will, this shows
		that the involution length of $\PU(n,1)$ is $4$ for every $n\ge2$.
		
		The core of the proof is a three-involution decomposition theorem for square
		roots: every $g\in\PU(n,1)$ has a square root of involution length at most
		$3$, and this uniform bound is optimal. The key construction is carried out
		first on a two- or three-dimensional indefinite block and then completed in
		higher dimensions by a necessary and sufficient spectral pairing criterion
		on the positive-definite orthogonal complement. The
		four-factor lower bound is provided by complex reflections in a point. As a
		consequence, every element of $\PU(n,1)$ is a single commutator, and the two
		elements in the commutator representation can be simultaneously reversed
		by the same nontrivial holomorphic involution. This extends the
		two-dimensional single-commutator result of Paupert--Will to every $n\ge2$
		and thereby verifies Djokovi\'{c}'s Conjecture~A for the family
		$\PU(n,1)$, $n\ge2$.
	\end{abstract}
	
	
	
	\section{Introduction}
	
	Involutions are among the fundamental geometric transformations in the
	isometry groups of symmetric spaces. Every point of a Riemannian symmetric
	space carries a corresponding central involution; in the connected case,
	these involutions and their products are closely related to the identity
	component of the isometry group. This leads naturally to a uniform
	decomposition problem: given an isometry group, does there exist a constant,
	independent of the group element, such that every element is a product of at
	most that many involutions? The least such uniform bound is usually called
	the involution length of the group. For complex hyperbolic space, if all
	factors are additionally required to be holomorphic isometries, the problem
	is constrained simultaneously by Hermitian geometry and projective
	structure, and is closely connected with strong reversibility and commutator
	representations.
	
	Let $\HC$ denote complex hyperbolic $n$-space and let $\PU(n,1)$ be its group
	of holomorphic isometries. Throughout the paper, we consider only
	holomorphic involutions lying in $\PU(n,1)$, and we write
	\[
	\mathcal I
	=
	\{I\in\PU(n,1):I^2=1,\ I\ne1\}.
	\]
	For $g\in\PU(n,1)$, define
	\[
	\linv(g)
	=
	\min\{k\ge0:g=I_1\cdots I_k,\ I_j\in\mathcal I\},
	\]
	where the identity is assigned length $0$. The involution length of the group
	is then
	\[
	\linv(\PU(n,1))
	=
	\sup_{g\in\PU(n,1)}\linv(g).
	\]
	If antiholomorphic involutions are allowed, the nature of the problem is
	substantially different: holomorphic complex hyperbolic isometries then admit
	much shorter decompositions into antiholomorphic involutions. For related
	results, see Gongopadhyay--Thomas~\cite{GongopadhyayThomas2016}. In this
	paper, however, all factors are required to be holomorphic involutions in
	$\PU(n,1)$.
	
	This problem is directly related to strong reversibility. An element of a
	group is a product of at most two involutions if and only if it is strongly
	reversible, that is, if an involution conjugates it to its inverse. Thus, the
	two-involution decomposition problem is equivalent to the strong
	reversibility problem. On the other hand, if
	\[
	a^2=b^2=e^2=1,
	\]
	then
	\[
	(abe)^2=[ab,eb]_{\mathrm c},
	\qquad
	[x,y]_{\mathrm c}=xyx^{-1}y^{-1}.
	\]
	Hence the square of a product of three involutions naturally gives a
	commutator representation. This elementary identity links involution length,
	square-root decompositions, and commutator representations, and forms the
	basic starting point of our proof.
	
	Such questions in isometry groups have a long history. Basmajian--Maskit
	\cite{BasmajianMaskit2012} studied involution length and commutator
	representations in spherical, Euclidean, and real hyperbolic spaces.
	Decompositions into involutions and reflections in general linear groups and
	indefinite unitary groups have also been studied systematically
	\cite{Radjavi1969,GustafsonHalmosRadjavi1976,
		KnuppelNielsen1991,DjokovicMalzan1982}. These results provide important
	algebraic background, but they do not directly determine the holomorphic
	involution length of $\PU(n,1)$: all linear lifts must preserve a prescribed
	Hermitian form of signature $(n,1)$, and their projectivizations must still
	define nontrivial holomorphic involutions. For reversibility and strong
	reversibility of complex hyperbolic isometries, and their relation to
	two-involution decompositions, see
	Gongopadhyay--Parker~\cite{GongopadhyayParker2013}.
	
	The holomorphic involution-length problem in complex hyperbolic dimension
	two was completely solved by Paupert--Will~\cite{PaupertWill2017}. They
	proved that
	\[
	\linv(\PU(2,1))=4,
	\]
	and determined exactly which complex hyperbolic isometries in dimension two
	can be written as products of three holomorphic involutions. Their proof
	combines explicit constructions in complex hyperbolic geometry with the
	product-map method on the space of conjugacy classes; in particular, the
	regular elliptic case depends on a detailed analysis of the elliptic
	conjugacy-class space of $\PU(2,1)$ and its reducible walls. They also proved
	that every element of $\PU(2,1)$ is a commutator and, in
	\cite[Section~8.2, proof of Proposition~20]{PaupertWill2017}, that every
	element of $\PU(2,1)$ has a square root of involution length at most $3$.
	
	The higher-dimensional case is more complicated. Paupert--Will observed
	that the refined method used to raise the lower bound to $4$ in dimension
	two relies on a detailed understanding of the chamber structure of the space
	of elliptic conjugacy classes, which becomes substantially more complicated
	in higher dimensions and is therefore difficult to generalize directly. For
	$n\ge3$, they obtained
	\[
	3\le\linv(\PU(n,1))\le8.
	\]
	On the other hand, B\"unger's decomposition result for unitary groups over
	the complex numbers \cite[Section~4, p.~98, case~\textup{(i)}]{Bunger1997}
	implies the projective upper bound $5$. Thus, before the present work, one had
	\begin{equation}\label{eq:previous-range}
		3\le\linv(\PU(n,1))\le5,
		\qquad n\ge3.
	\end{equation}
	Bhunia--Gongopadhyay~\cite{BhuniaGongopadhyay2020} also noted that the exact
	involution length of $\PU(n,1)$ remained open for $n\ge3$. Our first result
	resolves this problem.
	
	\begin{theorem}\label{thm:main}
		For every integer $n\ge3$,
		\[
		\linv(\PU(n,1))=4.
		\]
		Equivalently, every $g\in\PU(n,1)$ is a product of at most four
		holomorphic involutions, and some elements cannot be written as products
		of at most three holomorphic involutions.
	\end{theorem}
	
	Combining this with the two-dimensional result of Paupert--Will gives
	\begin{equation}\label{eq:all-dimensions-intro}
		\linv(\PU(n,1))=4,
		\qquad n\ge2.
	\end{equation}
	
	The higher-dimensional lower bound cannot be deduced merely from the
	embedding $\PU(2,1)\subset\PU(n,1)$, since the additional positive-definite
	directions may shorten an involution decomposition that was minimal in lower
	dimension. We therefore do not attempt to extend the two-dimensional chamber
	analysis of conjugacy classes. Instead, we use a higher-dimensional argument
	based on rank-one perturbations, spectral pairing, and dimension counts for
	eigenspaces. More precisely, we consider a family of complex reflections in
	a point (point rotations) that exists in every dimension. At the linear level,
	these elements are rank-one perturbations of scalar operators. If such a point
	rotation were a product of three holomorphic involutions, then, after
	separating one involution, the product of the remaining two involutions would,
	on the one hand, have prescribed eigenspaces of large dimension because of the
	rank-one perturbation structure and, on the other hand, have eigenvalues paired
	under inversion because of strong reversibility. For $n\ge4$, these two
	properties produce a contradiction by a dimension count. The case $n=3$ is
	precisely the critical dimension in which this estimate ceases to give a
	contradiction; this case requires a separate four-dimensional Hermitian
	argument.
	
	The four-factor upper bound follows from the following more structural
	square-root theorem. Instead of constructing a four-involution decomposition
	of $g$ directly, we study the smallest involution length attainable among all
	square roots
	\[
	h^2=g
	\]
	and prove that this square-root optimization problem has the same optimal
	uniform value in every dimension $n\ge2$.
	
	\begin{theorem}\label{thm:sqrt-main}
		For every $n\ge2$ and every $g\in\PU(n,1)$, there exists
		$h\in\PU(n,1)$ such that
		\[
		h^2=g,\qquad \linv(h)\le3.
		\]
		Moreover, this uniform upper bound is optimal:
		\[
		\sup_{g\in\PU(n,1)}
		\min_{\substack{h\in\PU(n,1)\\h^2=g}}
		\linv(h)
		=
		3.
		\]
	\end{theorem}
	
	For $n=2$, this square-root property was already proved by Paupert--Will
	\cite[Section~8.2, proof of Proposition~20]{PaupertWill2017}; our construction
	applies uniformly to every $n\ge2$ and covers elliptic and loxodromic
	isometries as well as all types of parabolic isometries.
	
	The proof of Theorem~\ref{thm:sqrt-main} first reduces the indefinite part
	needed for the square-root construction to a low-dimensional invariant block
	and then completes the construction through a spectral pairing
	mechanism on the positive-definite orthogonal complement. Since the
	Hermitian form has signature $(n,1)$, one may choose a two-dimensional
	nondegenerate invariant subspace of signature $(1,1)$ for elliptic,
	loxodromic, and two-step unipotent isometries, whereas a three-dimensional
	nondegenerate invariant subspace of signature $(2,1)$ suffices for
	three-step unipotent isometries. We refer to these low-dimensional subspaces
	as \emph{active blocks}.
	
	In the case of a two-dimensional active block, we first construct a square
	root and multiply it by a suitable unitary involution so that the resulting
	unitary transformation has a negative-type eigenvalue with phase $c$. In the
	three-step unipotent case, we first construct a loxodromic two-dimensional
	block having the same phase inside the three-dimensional active block. The
	remaining spectral pairing problem is then reduced to the positive-definite
	orthogonal complement. We prove that, after allowing a choice of signs for
	the square roots of the eigenvalues and multiplication by a further unitary
	involution, the only obstruction to making the resulting spectrum invariant
	under
	\[
	\zeta\longmapsto \frac{c^2}{\zeta}
	\]
	is a determinant phase condition. Combining this criterion with an
	appropriate choice of phase completes the three-involution square-root
	construction for every isometry type.
	
	Combining the square-root theorem with the identity above immediately yields
	a single-commutator representation; moreover, the
	two elements in the resulting commutator can be simultaneously reversed by
	the same nontrivial holomorphic involution.
	
	To fix notation, for an arbitrary group $G$ define
	\[
	\operatorname{cl}_G(g)
	=
	\inf\left\{
	k\in\mathbb Z_{\ge0}:
	g=\prod_{j=1}^{k}[x_j,y_j]_{\mathrm c},
	\ x_j,y_j\in G
	\right\},
	\]
	and
	\[
	\operatorname{cl}(G)
	=
	\sup_{g\in G}\operatorname{cl}_G(g),
	\]
	where the empty product is $1$ and $\inf\varnothing=+\infty$ by convention.
	
	Djokovi\'{c}~\cite[Conjecture~A]{Djokovic1986} conjectured that every
	centerless simple real Lie group has property~(C), meaning that every group
	element is a commutator. We verify this conjecture for the entire family of
	complex hyperbolic isometry groups $\PU(n,1)$, $n\ge2$, and obtain the
	following stronger simultaneous-reversal property.
	
	\begin{corollary}\label{cor:commutator-main}
		For every $n\ge2$ and every $g\in\PU(n,1)$, there exist
		$x,y\in\PU(n,1)$ and a nontrivial holomorphic involution $I$ such that
		\[
		g=[x,y]_{\mathrm c},
		\qquad
		IxI=x^{-1},
		\qquad
		IyI=y^{-1}.
		\]
		In particular,
		\[
		\operatorname{cl}(\PU(n,1))=1.
		\]
	\end{corollary}
	
	Thus, the three conclusions of this paper are not independent decomposition
	results, but are linked by a single structure. The rank-one perturbation and
	spectral pairing associated with point rotations give the exact four-factor
	lower bound; low-dimensional indefinite blocks and the necessary and
	sufficient spectral pairing criterion give optimal three-involution square
	roots; and the identity for the square of a product of three involutions turns
	these square roots into four-involution decompositions and
	single-commutator representations with simultaneous reversal. The
	optimality of the square-root constant $3$ follows from the closure of
	products of two involutions under squaring and the two-factor obstruction for
	point rotations.
	
	The paper is organized as follows. Section~\ref{sec:prelim} fixes the
	Hermitian model of complex hyperbolic space and the basic notation, and
	treats projective lifts, normalization of involutions, the spectral criterion
	for strong reversibility, and the decomposition into low-dimensional active
	blocks in Hermitian signature $(n,1)$. Section~\ref{sec:lower} proves the
	four-factor lower bound using point rotations. Section~\ref{sec:sqrt-tools}
	develops the algebraic reductions and phase-selection tools needed for the
	square-root construction and establishes the necessary and sufficient
	spectral pairing criterion on the positive-definite orthogonal complement.
	Section~\ref{sec:sqrt-types} treats the two-dimensional active blocks and the
	three-dimensional active block of a three-step unipotent isometry. Finally,
	Section~\ref{sec:sqrt-conclusion} proves the two main theorems and
	Corollary~\ref{cor:commutator-main}.
	
	\section{Preliminaries}\label{sec:prelim}
	
	We fix the Hermitian linear-algebra notation used throughout the paper. We
	then describe the standard form of holomorphic involutions, orthogonal
	gluing, and the relation between two-involution decompositions and strong
	reversibility, before giving a decomposition of isometries into
	low-dimensional active blocks in Hermitian signature $(n,1)$. Let
	$V=\C^{n,1}=\C^{n+1}$ be endowed with a nondegenerate Hermitian form
	$\ip{\cdot}{\cdot}$ of signature $(n,1)$, and write
	\[
	\begin{aligned}
		V^-&=\{Z\in V:\ip{Z}{Z}<0\},\\
		V^0&=\{Z\in V\setminus\{0\}:\ip{Z}{Z}=0\},\\
		V^+&=\{Z\in V:\ip{Z}{Z}>0\}.
	\end{aligned}
	\]
	If $\pi:V\setminus\{0\}\to\CP^n$ denotes projectivization, then
	$\HC=\pi(V^-)$. Let $\UU(n,1)$ be the unitary group preserving this
	Hermitian form, and let
	\[
	\PU(n,1)
	=
	\UU(n,1)\big/\{e^{i\theta}\Id:\theta\in\R\}
	\]
	be the corresponding projective unitary group. It acts naturally on $\HC$
	and is the group of holomorphic isometries of $\HC$. For $A\in\UU(n,1)$,
	write $[A]$ for its projective class in $\PU(n,1)$. The notation $\spec(A)$
	always denotes the spectral multiset counted with algebraic multiplicity, and
	$\sqcup$ denotes multiset union with multiplicities retained.
	
	We take the Hermitian form $\ip{\cdot}{\cdot}$ to be conjugate-linear in the
	first variable and linear in the second, and write
	\[
	\T=\{z\in\C:|z|=1\}.
	\]
	
	For a nondegenerate Hermitian subspace $E\le V$, the notation $E>0$
	(respectively, $E<0$) means that the Hermitian form is positive definite
	(respectively, negative definite) on $E$. Since
	\[
	\operatorname{sign}(V)=(n,1),
	\]
	the Witt index of $V$ is $1$; equivalently, every totally isotropic subspace
	has complex dimension at most $1$. See Goldman~\cite{Goldman1999} for the
	basic background.
	
	We next introduce notation for involution length that will be used below.
	For fixed $n$, retain the notation $\mathcal I$ and $\linv$ from the
	introduction. For $k\ge0$, set
	\[
	\mathcal I^{\le k}
	=
	\left\{
	I_1\cdots I_j:
	j\in\{0,\ldots,k\},\
	I_1,\ldots,I_j\in\mathcal I
	\right\},
	\]
	where the product for $j=0$ is understood to be the identity. Thus,
	$g\in\mathcal I^{\le k}$ if and only if $\linv(g)\le k$.
	
	We begin with a projective lifting fact that will be used repeatedly.
	
	\begin{lemma}\label{lem:projective-lifting}
		Let $A\in\UU(n,1)$.
		\begin{enumerate}[label=\textup{(\roman*)}]
			\item If $[A]^2=1$, then there exists $\tau\in\T$ such that
			$J=\tau A$ satisfies $J^2=\Id$.
			\item If $[A]\in\mathcal I^{\le2}$, then there exist a unitary
			involution $J\in\UU(n,1)$ and $z\in\T$ such that
			\begin{equation}\label{eq:projective-strong-lift}
				JAJ=zA^{-1}.
			\end{equation}
		\end{enumerate}
	\end{lemma}
	
	\begin{proof}
		If $[A]^2=1$, then $A^2=\xi\Id$. Since $A$ is unitary,
		$\xi\in\T$. Choose $\tau\in\T$ such that $\tau^2\xi=1$; this
		proves~\textup{(i)}.
		
		To prove~\textup{(ii)}, pad a projective decomposition of length less
		than two with the identity and use~\textup{(i)} to choose linear unitary
		involutions $J_1,J_2$ such that
		\[
		[J_1J_2]=[A].
		\]
		There is then an $\alpha\in\T$ such that $J_1J_2=\alpha A$. Since
		\[
		J_1(\alpha A)J_1
		=J_2J_1=(\alpha A)^{-1},
		\]
		we obtain
		\[
		J_1AJ_1=\alpha^{-2}A^{-1}.
		\]
		Taking $J=J_1$ and $z=\alpha^{-2}$ proves the assertion.
	\end{proof}
	
	We now put linear unitary involutions into a standard form that will be
	convenient below.
	
	\begin{lemma}\label{lem:normalise-involution}
		Let $J\in\UU(n,1)$ satisfy $J^2=\Id$. Then there exists
		$\varepsilon\in\{\pm1\}$ such that the $(-1)$-eigenspace $E$ of
		$\varepsilon J$ is positive definite and
		\[
		\varepsilon J=\Id-2P_E,
		\]
		where $P_E$ is the Hermitian orthogonal projection onto $E$.
	\end{lemma}
	
	\begin{proof}
		Since $J^2=\Id$, the map $J$ is diagonalizable and its eigenvalues are
		$\pm1$. Hence
		\[
		V=V_+(J)\oplus V_-(J),
		\qquad
		V_\pm(J)=\ker(J\mp\Id).
		\]
		If $x\in V_+(J)$ and $y\in V_-(J)$, then
		\[
		\ip{x}{y}=\ip{Jx}{Jy}=-\ip{x}{y},
		\]
		so $V_+(J)\perp V_-(J)$. Since $V$ is nondegenerate, both terms in
		this orthogonal direct sum are nondegenerate. Their negative indices sum
		to $1$, and therefore one of them is positive definite. If
		$V_-(J)>0$, take $\varepsilon=1$. If $V_-(J)$ is not positive
		definite, then $V_+(J)>0$, and we take $\varepsilon=-1$. Thus
		\[
		E:=V_-(\varepsilon J)>0.
		\]
		With respect to the orthogonal decomposition $V=E^\perp\perp E$, the
		map $\varepsilon J$ is $\Id$ on $E^\perp$ and $-\Id$ on $E$.
		Consequently, if $Z=Z_{E^\perp}+Z_E$, then $P_EZ=Z_E$, and
		\[
		\varepsilon JZ
		=Z_{E^\perp}-Z_E
		=Z-2P_EZ.
		\]
		Thus $\varepsilon J=\Id-2P_E$.
	\end{proof}
	
	For every positive-definite subspace $E>0$, we henceforth write
	\[
	I_E:=\Id-2P_E.
	\]
	By Lemma~\ref{lem:projective-lifting}\textup{(i)} and
	Lemma~\ref{lem:normalise-involution}, every nontrivial projective involution
	can be represented as $[I_E]$. We call such a linear unitary involution
	$I_E$ a \emph{standard involution}.
	
	We next record the orthogonal gluing fact used repeatedly below.
	
	\begin{lemma}\label{lem:orthogonal-gluing}
		Suppose that
		\[
		V=V_1\orth\cdots\orth V_s,
		\qquad
		A=A_1\orth\cdots\orth A_s\in\UU(V),
		\]
		where each $V_\nu$ is nondegenerate. Suppose that, for some integer
		$k$, every $A_\nu$ has a decomposition
		\[
		A_\nu=J_{\nu,1}\cdots J_{\nu,k},
		\qquad J_{\nu,j}\in\UU(V_\nu),\quad J_{\nu,j}^2=\Id,
		\]
		where shorter decompositions may be padded with identity maps. Then
		\[
		A=\widehat J_1\cdots\widehat J_k,
		\qquad
		\widehat J_j=J_{1,j}\orth\cdots\orth J_{s,j},
		\]
		and every $\widehat J_j$ is a linear unitary involution. After
		projectivizing and deleting factors whose projective class is the
		identity,
		\[
		\linv([A])\le k.
		\]
		Moreover, if every $A_\nu$ is semisimple, then $A$ is semisimple.
	\end{lemma}
	
	\begin{proof}
		Orthogonal direct sums preserve unitarity, and
		\[
		\widehat J_j^2
		=J_{1,j}^2\orth\cdots\orth J_{s,j}^2
		=\Id.
		\]
		Thus every $\widehat J_j$ is a linear unitary involution, and
		\[
		\widehat J_1\cdots\widehat J_k
		=(J_{1,1}\cdots J_{1,k})\orth\cdots\orth
		(J_{s,1}\cdots J_{s,k})
		=A.
		\]
		If $[\widehat J_j]\ne1$, then Lemma~\ref{lem:normalise-involution}
		shows that its projective class can be represented by a standard
		involution. If $\widehat J_j=\pm\Id$, then its projective class is the
		identity and can be deleted. Hence $\linv([A])\le k$. Finally, a finite
		direct sum of semisimple operators remains diagonalizable, and is
		therefore semisimple.
	\end{proof}
	
	We now recall the equivalence between products of two involutions and strong
	reversibility, and give the strong-reversibility criterion needed below.
	
	\begin{definition}
		An operator $A$ is called \emph{strongly reversible} if an involution
		conjugates $A$ to $A^{-1}$, that is, if there exists $J$ with $J^2=\Id$
		such that $JAJ=A^{-1}$.
	\end{definition}
	
	\begin{lemma}\label{lem:two-strong}
		If $A=I_1I_2$, where $I_1^2=I_2^2=\Id$, then
		$I_1AI_1=A^{-1}$. Conversely, if $JAJ=A^{-1}$, then
		\[
		A=J(JA)
		\]
		is a product of two involutions.
	\end{lemma}
	
	\begin{proof}
		The first identity follows from
		$I_1(I_1I_2)I_1=I_2I_1$. Conversely, it suffices to observe that
		$(JA)^2=JAJA=A^{-1}A=\Id$.
	\end{proof}
	
	\begin{lemma}\label{lem:spectral-strong}
		Let $A\in\UU(n,1)$ be semisimple and admit an $A$-invariant
		orthogonal decomposition
		\[
		V=L\perp P,
		\qquad \dim_{\C}L=1,\quad L<0,\quad P>0.
		\]
		Suppose that $A|_L=c\Id_L$, where $c\in\T$. If the eigenvalues of
		$A|_P$, counted with multiplicity, are invariant under
		\[
		\rho_c(\zeta)=\frac{c^2}{\zeta},
		\]
		then $[A]$ is a product of at most two nontrivial holomorphic
		involutions.
	\end{lemma}
	
	\begin{proof}
		Set $U=c^{-1}A$. Then $U|_L=\Id_L$. If $\zeta$ is an eigenvalue of
		$A|_P$, then $\zeta/c$ is an eigenvalue of $U|_P$, and
		\[
		\frac{\rho_c(\zeta)}{c}=\frac{c}{\zeta}
		=\left(\frac{\zeta}{c}\right)^{-1}.
		\]
		Thus, the spectrum of $U|_P$ occurs in pairs
		$w\leftrightarrow w^{-1}$.
		
		Since $U$ is semisimple, for each pair of distinct spectral values
		$\{w,w^{-1}\}$ choose a unitary isometry between the corresponding
		equidimensional positive-definite eigenspaces and interchange them. Take
		the identity on the eigenspaces for the eigenvalues $\pm1$ and on the
		negative line $L$. Taking the orthogonal direct sum of these maps, as in
		Lemma~\ref{lem:orthogonal-gluing}, gives a unitary involution $I_1$
		satisfying
		\[
		I_1UI_1=U^{-1}.
		\]
		Let $I_2=I_1U$. Then
		\[
		I_2^2=I_1UI_1U=\Id,
		\qquad U=I_1I_2.
		\]
		Both factors fix the negative line $L$, so their $(-1)$-eigenspaces
		are contained in the positive-definite subspace $P$. After projective
		normalization and deletion of any identity factors, this gives a
		decomposition of $[A]=[U]$ into at most two holomorphic involutions.
	\end{proof}
	
	Finally, we give the active-block classification used throughout the paper.
	The essential point is that Witt index $1$ forces every noncompact or
	nonsemisimple part to be concentrated in a nondegenerate invariant subspace
	of dimension at most three.
	
	\begin{lemma}\label{lem:active-blocks}
		Every nonidentity element $[T]\in\PU(n,1)$ belongs to one of the
		following cases:
		\begin{enumerate}[label=\textup{(\alph*)}]
			\item $T$ is elliptic: it is semisimple and admits an
			$\ip{\cdot}{\cdot}$-orthogonal eigenbasis;
			
			\item $T$ is loxodromic: there is an orthogonal decomposition
			\[
			V=H\orth W,
			\qquad \operatorname{sign}(H)=(1,1),\quad W>0,
			\]
			such that $T|_H$ is a semisimple loxodromic block and $T|_W$ is
			semisimple;
			
			\item $T$ is two-step unipotent: there is an orthogonal
			decomposition
			\[
			V=H\orth W,
			\qquad \operatorname{sign}(H)=(1,1),\quad W>0,
			\]
			and, in a suitable null basis,
			\[
			T|_H=\lambda
			\begin{pmatrix}
				1&it\\
				0&1
			\end{pmatrix},
			\qquad |\lambda|=1,\quad t\in\R\setminus\{0\},
			\]
			while $T|_W$ is semisimple;
			
			\item $T$ is three-step unipotent: after multiplication by a
			central scalar, one can choose the lift $T$ so that there is an
			orthogonal decomposition
			\[
			V=H_3\orth W,
			\qquad
			\operatorname{sign}(H_3)=(2,1),\quad W>0,
			\]
			with
			\[
			T=N\orth D,
			\qquad
			(N-\Id)^3=0\ne(N-\Id)^2,
			\]
			where $D\in\UU(W)$ is semisimple.
		\end{enumerate}
	\end{lemma}
	
	\begin{proof}
		First suppose that $T$ is semisimple. If $Tv=\alpha v$ and
		$Tw=\beta w$, then
		\begin{equation}\label{eq:eigen-orthogonality}
			\ip{v}{w}
			=\ip{Tv}{Tw}
			=\overline\alpha\beta\,\ip{v}{w}.
		\end{equation}
		If all eigenvalues lie in $\T$, then eigenspaces corresponding to
		distinct eigenvalues are mutually orthogonal. Since $T$ is semisimple,
		their direct sum is $V$; since $V$ is nondegenerate, each eigenspace is
		nondegenerate. A Hermitian orthogonal basis may therefore be chosen in
		each eigenspace, which gives~\textup{(a)}.
		
		Suppose now that $|\alpha|\ne1$. For arbitrary $v,w\in E_\alpha$,
		Equation~\eqref{eq:eigen-orthogonality} gives
		$\ip{v}{w}=|\alpha|^2\ip{v}{w}$, so $E_\alpha$ is totally
		isotropic. Since the Witt index is $1$, we have
		$\dim E_\alpha=1$. By nondegeneracy of the Hermitian form, there are
		an eigenspace $E_\beta$ and vectors $v\in E_\alpha$ and
		$w\in E_\beta$ such that $\ip{v}{w}\ne0$.
		Equation~\eqref{eq:eigen-orthogonality} then implies
		$\beta=\overline\alpha^{-1}$, and this eigenspace is again
		one-dimensional. Hence
		\[
		H=E_\alpha\oplus E_{\overline\alpha^{-1}}
		\]
		is a Hermitian plane of signature $(1,1)$. If there were another pair
		of eigenvalues off the unit circle, it would yield another null line
		orthogonal to $E_\alpha$, producing a two-dimensional totally
		isotropic subspace and contradicting the fact that the Witt index is
		$1$. Therefore $H^\perp>0$, which proves~\textup{(b)}.
		
		It remains to consider the case in which $T$ is not semisimple. Take
		the multiplicative Jordan--Chevalley decomposition
		\[
		T=SU=US,
		\]
		where $S$ is semisimple and $U$ is unipotent. Let $A^\dagger$ denote
		the adjoint with respect to $\ip{\cdot}{\cdot}$, so that
		$\ip{Ax}{y}=\ip{x}{A^\dagger y}$. Since $T\in\UU(n,1)$, we have
		$T^\dagger=T^{-1}$. The uniqueness of the Jordan--Chevalley
		decomposition then gives
		\[
		S^\dagger=S^{-1},\qquad U^\dagger=U^{-1}.
		\]
		Set
		\[
		Y=\log U
		=\sum_{j\ge1}\frac{(-1)^{j+1}}{j}(U-\Id)^j.
		\]
		This sum is finite because $U-\Id$ is nilpotent. Thus $Y$ is
		nilpotent, $SY=YS$, and
		\begin{equation}\label{eq:Y-skew}
			Y^\dagger=-Y,
			\qquad
			\ip{Yx}{y}=-\ip{x}{Yy}.
		\end{equation}
		
		Consider $K=\im Y\cap\ker Y$. If $x=Ya$ and $y=Yb$ belong to $K$,
		then
		\[
		\ip{x}{y}=\ip{Ya}{Yb}=-\ip{a}{Y^2b}=0.
		\]
		Hence $K$ is totally isotropic. For every nontrivial Jordan block of
		$Y$, the space $\im Y\cap\ker Y$ restricted to that block is
		one-dimensional, whereas a trivial Jordan block contributes nothing to
		this intersection. Consequently,
		$\dim_{\C}(\im Y\cap\ker Y)$ equals the number of nontrivial Jordan
		blocks of $Y$. Since the Witt index is $1$, the operator $Y$ has exactly
		one nontrivial Jordan block.
		
		Let $q$ be the length of this block and choose a cyclic vector $u$ such
		that
		\[
		Y^{q-1}u\ne0,\qquad Y^qu=0.
		\]
		Repeated use of~\eqref{eq:Y-skew} gives
		\begin{equation}\label{eq:Y-chain-pairing}
			\ip{Y^iu}{Y^ju}=(-1)^i\ip{u}{Y^{i+j}u}.
		\end{equation}
		If $q\ge4$, then $Y^{q-1}u$ and $Y^{q-2}u$ are linearly independent,
		while~\eqref{eq:Y-chain-pairing} shows that their norms and mutual
		inner product all vanish. They would therefore span a two-dimensional
		totally isotropic subspace, a contradiction. Thus $q\in\{2,3\}$.
		
		Since $SY=YS$ and $S$ is semisimple, $Y$ preserves every eigenspace
		of $S$. Write $V=\bigoplus_\lambda E_\lambda(S)$ and
		$u=\sum_\lambda u_\lambda$. Since $Y^{q-1}u\ne0$, there is at least
		one $\lambda$ for which $Y^{q-1}u_\lambda\ne0$. Replacing the original
		cyclic vector $u$ by this $u_\lambda$, we may arrange that the unique
		nontrivial Jordan chain lies entirely in $E_\lambda(S)$. Hence, on the
		active block
		\[
		H_q=\operatorname{span}\{u,Yu,\ldots,Y^{q-1}u\},
		\]
		we have $S=\lambda\Id$.
		
		Suppose first that $q=2$, and set $e_0=Yu$. Then $Ye_0=0$, and
		Equation~\eqref{eq:Y-skew} implies that $e_0\perp\ker Y$. Since
		$V=\C u\oplus\ker Y$ and $e_0\ne0$, nondegeneracy gives
		$\ip{e_0}{u}\ne0$. Moreover,
		$\ip{e_0}{u}=-\overline{\ip{e_0}{u}}$, so this inner product is
		nonzero and purely imaginary. Choosing $s\in\C$ appropriately and
		replacing $u$ by $u+se_0$, we may also make $u$ a null vector without
		changing the Jordan-chain relation $Yu=e_0$. It follows that
		\[
		H=\operatorname{span}\{e_0,u\}
		\]
		is a Hermitian plane of signature $(1,1)$. In a normalized null basis,
		\[
		U|_H=
		\begin{pmatrix}
			1&it\\
			0&1
		\end{pmatrix},
		\qquad t\in\R\setminus\{0\}.
		\]
		Since $S|_H=\lambda\Id$ and $S$ is unitary, $|\lambda|=1$. The unique
		nontrivial Jordan block is already contained in $H$, so
		$Y|_{H^\perp}=0$ and hence $T|_{H^\perp}=S|_{H^\perp}$ is
		semisimple. This proves~\textup{(c)}.
		
		Suppose now that $q=3$, and set
		\[
		a=\ip{Y^2u}{u}.
		\]
		Then $Y^2u\perp\ker Y$ and $\ip{Y^2u}{Yu}=0$. Moreover,
		$V=\C u+\C Yu+\ker Y$. If $a=0$, the nonzero vector $Y^2u$ would be
		orthogonal to $u$, $Yu$, and $\ker Y$, contradicting nondegeneracy of
		the Hermitian form. Thus $a\ne0$. Since
		$(Y^2)^\dagger=Y^2$, we have $a\in\R$, and
		\[
		\ip{Yu}{Yu}=-a.
		\]
		In the basis $(Y^2u,Yu,u)$, the Gram matrix is
		\[
		\begin{pmatrix}
			0&0&a\\
			0&-a&*\\
			a&*&*
		\end{pmatrix},
		\]
		whose determinant is $a^3\ne0$. Therefore
		\[
		H_3=\operatorname{span}\{u,Yu,Y^2u\}
		\]
		is nondegenerate. The space $H_3$ contains a null line, and $V$ has
		negative index $1$; hence $\operatorname{sign}(H_3)=(2,1)$ and
		$H_3^\perp>0$. On $H_3$ we have $T=\lambda U$ with
		$|\lambda|=1$. After multiplying the entire lift by $\lambda^{-1}$,
		the restriction to the active block becomes a pure three-step unipotent
		operator $N$ satisfying
		\[
		(N-\Id)^3=0\ne(N-\Id)^2.
		\]
		Again, the unique nontrivial Jordan block is contained in $H_3$, so the
		restriction to the positive-definite orthogonal complement is
		semisimple. This proves~\textup{(d)}.
	\end{proof}
	
	\section{Point Rotations and the Four-Factor Lower Bound}\label{sec:lower}
	
	We use complex reflections in a point to construct the four-factor lower
	bound. Fix $p\in V^-$ and let $P_p$ be the Hermitian orthogonal projection
	onto the negative line $\C p$. For $\lambda\in\T$, define
	\begin{equation}\label{eq:point-rotation}
		A_\lambda=\lambda\Id+(1-\lambda)P_p.
	\end{equation}
	With respect to the orthogonal decomposition
	\[
	V=p^\perp\perp \C p,
	\]
	we have $P_p|_{p^\perp}=0$ and $P_p|_{\C p}=\Id$, and therefore
	\[
	A_\lambda|_{p^\perp}=\lambda\Id,
	\qquad
	A_\lambda|_{\C p}=\Id.
	\]
	Consequently,
	\[
	\spec(A_\lambda)=\{\lambda^{(n)},1^{(1)}\},
	\qquad
	\det A_\lambda=\lambda^n.
	\]
	When $\lambda\ne1$, the $\lambda$-eigenspace $p^\perp$ is positive
	definite, whereas the $1$-eigenspace $\C p$ is a negative line.
	
	We first establish a two-factor obstruction valid for every $n\ge2$. It will
	also be used in Section~\ref{sec:sqrt-conclusion} to prove optimality of the
	square-root constant.
	
	\begin{lemma}\label{lem:point-not-two}
		For every $n\ge2$, if $\lambda^2\ne1$, then
		$[A_\lambda]\in\PU(n,1)$ is not a product of two involutions.
	\end{lemma}
	
	\begin{proof}
		If $[A_\lambda]$ were a product of two involutions, then
		Lemma~\ref{lem:projective-lifting}\textup{(ii)} would give a unitary
		involution $J$ and $z\in\T$ such that
		\[
		JA_\lambda J=zA_\lambda^{-1}.
		\]
		Unitary conjugation preserves the Hermitian type of each eigenspace.
		Thus the unique negative-type eigenline on the left still corresponds
		to the eigenvalue $1$, whereas the eigenvalue on the negative line
		$\C p$ on the right is $z$. Hence $z=1$. It follows that
		$A_\lambda$ and $A_\lambda^{-1}$ have the same positive-type
		spectrum, so
		\[
		\lambda=\lambda^{-1},
		\]
		contrary to $\lambda^2\ne1$.
	\end{proof}
	
	For $n\ge4$, the rank-one perturbation structure of $A_\lambda$ can be
	combined directly with the spectral pairing property of a product of two
	involutions to rule out a three-factor decomposition. The following argument
	uses only geometric multiplicities and does not require the unitary
	transformations involved to be semisimple.
	
	\begin{proposition}\label{prop:lower-high}
		If $n\ge4$ and $\lambda^4\ne1$, then $[A_\lambda]$ is not a product of
		at most three projective holomorphic involutions.
	\end{proposition}
	
	\begin{proof}
		Suppose, to the contrary, that $[A_\lambda]$ is a product of at most
		three projective holomorphic involutions. Applying
		Lemma~\ref{lem:projective-lifting}\textup{(i)} to each factor, lift
		each projective involution to a linear unitary involution, and denote the
		three lifts by $I_1,I_2,J$. If there are fewer than three factors, pad
		the decomposition with identity maps only at the linear level. Thus
		$I_1^2=I_2^2=J^2=\Id$, and there exists $\zeta\in\T$ such that
		\[
		X:=I_1I_2J=\zeta A_\lambda.
		\]
		Every linear unitary involution has determinant $\pm1$, so
		$\det X=\pm1$. Put $d=n+1$ and $\beta=\zeta\lambda$. Since
		$\det A_\lambda=\lambda^{d-1}$,
		\[
		\beta^d=(\zeta\lambda)^d=\pm\lambda.
		\]
		If $\beta^4=1$, then $\lambda=\pm\beta^d$ would also imply
		$\lambda^4=1$, contrary to the assumption. Therefore
		$\beta^4\ne1$.
		
		Set
		\[
		D=XJ=I_1I_2.
		\]
		Using $A_\lambda=\lambda\Id+(1-\lambda)P_p$, we may write
		\[
		X=\beta\Id+(\zeta-\beta)P_p,
		\qquad
		D=\beta J+(\zeta-\beta)P_pJ.
		\]
		Let
		\[
		a=\dim V_+(J),\qquad b=\dim V_-(J),\qquad a+b=d.
		\]
		If $v\in V_+(J)\cap p^\perp$, then $Dv=Xv=\beta v$. If
		$v\in V_-(J)\cap p^\perp$, then $Dv=-Xv=-\beta v$. Since
		$p^\perp$ has codimension $1$,
		\[
		\dim\bigl(V_\pm(J)\cap p^\perp\bigr)
		\ge \dim V_\pm(J)-1.
		\]
		Writing $m_D(t)=\dim\kerx(D-t\Id)$, we obtain
		\[
		m_D(\beta)\ge a-1,
		\qquad
		m_D(-\beta)\ge b-1,
		\]
		and hence
		\begin{equation}\label{eq:rank-one-multiplicity-lower}
			m_D(\beta)+m_D(-\beta)\ge d-2.
		\end{equation}
		On the other hand, $D=I_1I_2$ gives
		\[
		I_1DI_1=D^{-1}.
		\]
		Thus the eigenvalues $t$ and $t^{-1}$ have the same geometric
		multiplicity. In particular,
		\[
		m_D(\beta^{-1})=m_D(\beta),
		\qquad
		m_D(-\beta^{-1})=m_D(-\beta).
		\]
		Since $\beta^4\ne1$, the four numbers
		$\beta,-\beta,\beta^{-1},-\beta^{-1}$ are pairwise distinct, and the
		corresponding four eigenspaces form a direct sum. Combining this with
		Equation~\eqref{eq:rank-one-multiplicity-lower}, we get
		\[
		d\ge
		m_D(\beta)+m_D(-\beta)+m_D(\beta^{-1})+m_D(-\beta^{-1})
		\ge2(d-2).
		\]
		Therefore $d\le4$, contradicting $d=n+1\ge5$.
	\end{proof}
	
	The preceding dimension estimate gives no contradiction when $n=3$, so
	this critical dimension must be treated separately. In this case
	$\dim_{\C}V=4$. The standard involutions $I_E=\Id-2P_E$ fall into three
	types according to
	\[
	k=\dim E\in\{1,2,3\};
	\]
	we call such an involution an involution of type $k$. For
	$\lambda\ne\pm1$, the element $[A_\lambda]$ is neither the identity nor an
	involution, and Lemma~\ref{lem:point-not-two} also rules out a two-factor
	decomposition. It therefore remains only to analyze the possible types of
	three nontrivial involution factors.
	
	\begin{proposition}\label{prop:lower-n3}
		If $\lambda\ne\pm1$, then $[A_\lambda]\in\PU(3,1)$ is not a product
		of at most three holomorphic involutions.
	\end{proposition}
	
	\begin{proof}
		Since $\lambda\ne\pm1$, the element $[A_\lambda]$ is neither the
		identity nor an involution, and Lemma~\ref{lem:point-not-two} rules out
		a two-factor decomposition. It is therefore enough to rule out a
		decomposition into exactly three nontrivial holomorphic involutions.
		Suppose, to the contrary, that
		\begin{equation}\label{eq:n3triple}
			I_1I_2I_3=\zeta A_\lambda,
			\qquad \zeta\in\T.
		\end{equation}
		We first choose linear involution lifts by
		Lemma~\ref{lem:projective-lifting}\textup{(i)} and then normalize them
		to standard linear lifts using Lemma~\ref{lem:normalise-involution};
		the central scalars $-1$ introduced by this normalization are absorbed
		into $\zeta$.
		
		Suppose first that at least one of the three factors is of type $2$,
		and denote it by
		\[
		J=\Id-2P_E,
		\qquad \dim E=2.
		\]
		A cyclic permutation of the three factors only conjugates their total
		product, while a unitary conjugate of a point rotation with parameter
		$\lambda$ is still a point rotation with parameter $\lambda$, with only
		its fixed negative line changed. After renaming this negative line
		$\C p$, we may therefore assume that $J=I_3$. Put $D=I_1I_2$. Then
		\[
		D=\zeta A_\lambda J,
		\]
		and
		\[
		\tr D=\zeta\tr(A_\lambda J),
		\qquad
		\tr D^{-1}=\zeta^{-1}\tr(A_\lambda^{-1}J).
		\]
		Since $D$ is a product of two involutions, it is similar to $D^{-1}$,
		so these two traces are equal.
		
		A type-$2$ involution in four dimensions has eigenvalues
		$1,1,-1,-1$, and hence $\tr J=0$ and $\det J=1$. Set
		\[
		\kappa:=\tr(P_pJ).
		\]
		From
		\[
		A_\lambda=\lambda\Id+(1-\lambda)P_p,
		\qquad
		A_\lambda^{-1}=\lambda^{-1}\Id+(1-\lambda^{-1})P_p,
		\]
		we obtain
		\[
		\tr(A_\lambda J)=(1-\lambda)\kappa,
		\qquad
		\tr(A_\lambda^{-1}J)=(1-\lambda^{-1})\kappa.
		\]
		Moreover,
		\[
		\kappa=1-2\tr(P_pP_E)
		=1-2\frac{\langle P_Ep,P_Ep\rangle}{\langle p,p\rangle}\ge1,
		\]
		because $E>0$ and $p<0$. Thus $\kappa\ne0$. Since $\lambda\ne1$
		and the two traces are equal, we obtain
		\[
		\zeta^2=-\lambda^{-1},
		\qquad
		\zeta^4=\lambda^{-2}.
		\]
		On the other hand,
		\[
		\det D
		=\zeta^4\det(A_\lambda)\det J
		=\zeta^4\lambda^3
		=\lambda.
		\]
		But $D$ is a product of two linear involutions, so
		$\det D\in\{\pm1\}$. Hence $\lambda=\pm1$, a contradiction.
		
		Now suppose that none of the three factors is of type $2$. A type-$1$
		factor is a reflection in a positive line, whereas multiplication of a
		type-$3$ factor by the central scalar $-1$ turns it into a reflection
		in a negative line. Absorbing these central minus signs into $\zeta$,
		we may rewrite~\eqref{eq:n3triple} as
		\begin{equation}\label{eq:rankone-triple}
			R_1R_2R_3=\zeta A_\lambda,
			\qquad
			R_j=\Id-2P_{q_j},
			\quad \langle q_j,q_j\rangle\ne0.
		\end{equation}
		Each $R_j$ fixes the hyperplane $q_j^\perp$ pointwise. Since $V$ has
		complex dimension four, the intersection of these three hyperplanes,
		\[
		L=q_1^\perp\cap q_2^\perp\cap q_3^\perp,
		\]
		has dimension at least one. Choose $0\ne v\in L$. Then
		$R_1R_2R_3v=v$. With respect to
		\[
		V=p^\perp\perp\C p,
		\]
		write
		\[
		v=v_++v_-,
		\qquad v_+\in p^\perp,
		\quad v_-\in\C p.
		\]
		Equation~\eqref{eq:rankone-triple} and
		$A_\lambda v=\lambda v_++v_-$ give
		\begin{equation}\label{eq:v-components}
			(\zeta\lambda-1)v_++(\zeta-1)v_-=0.
		\end{equation}
		
		If $v$ is a nonzero null vector, then both $v_+$ and $v_-$ are
		nonzero. The orthogonal direct sum decomposition and
		Equation~\eqref{eq:v-components} force
		\[
		\zeta\lambda=1,
		\qquad \zeta=1,
		\]
		so $\lambda=1$, a contradiction.
		
		If $v$ is positive, then $v_+\ne0$, and hence
		$\zeta\lambda=1$. Every rank-one reflection has determinant $-1$;
		therefore
		\[
		-1=\det(R_1R_2R_3)=\zeta^4\lambda^3=\lambda^{-1},
		\]
		which gives $\lambda=-1$, again a contradiction.
		
		Finally, suppose that $v$ is negative. Then $v_-\ne0$, so
		Equation~\eqref{eq:v-components} gives $\zeta=1$; since
		$\lambda\ne1$, it also gives $v_+=0$. Thus $v\in\C p$, and hence
		$p\in L$. If some $q_j$ were negative, then $q_j^\perp$ would be a
		three-dimensional positive-definite subspace and could not contain the
		negative vector $p$. Thus every $q_j$ is positive and
		$q_j\in p^\perp$. Restricting~\eqref{eq:rankone-triple} to the
		positive-definite three-dimensional space $p^\perp$, we obtain
		\[
		R_1R_2R_3=\lambda\Id_{p^\perp},
		\qquad
		R_1R_2=\lambda R_3.
		\]
		Since $q_1^\perp\cap q_2^\perp\cap p^\perp$ has dimension at least
		one, there is a nonzero vector fixed by both $R_1$ and $R_2$, and hence
		$1$ is an eigenvalue of $R_1R_2$. On the other hand, the spectrum of
		$R_3$ on $p^\perp$ is $\{-1,1,1\}$, so the spectrum of
		$\lambda R_3$ is
		\[
		\{-\lambda,\lambda,\lambda\}.
		\]
		It follows that $1\in\{-\lambda,\lambda\}$, and once again
		$\lambda=\pm1$, a contradiction.
		
		Both possibilities have been excluded, so no three-factor
		decomposition exists.
	\end{proof}
	
	In summary, the case $n\ge4$ is handled by the rank-one perturbation and
	spectral pairing argument, whereas the critical case $n=3$ is treated
	separately by the four-dimensional Hermitian argument above. This gives a
	uniform four-factor lower bound.
	
	\begin{corollary}\label{cor:lower}
		For example, take $\lambda=e^{i\pi/3}$. Then, for every $n\ge3$,
		\[
		\linv([A_\lambda])\ge4.
		\]
		Consequently, $\linv(\PU(n,1))\ge4$.
	\end{corollary}
	
	\begin{proof}
		We have $\lambda^4\ne1$ and $\lambda\ne\pm1$. Apply
		Proposition~\ref{prop:lower-high} when $n\ge4$ and
		Proposition~\ref{prop:lower-n3} when $n=3$.
	\end{proof}
	
	\section{Three-Involution Square Roots and the Spectral Pairing Criterion}
	\label{sec:sqrt-tools}
	
	This section develops the algebraic and spectral tools needed below. The main
	strategy is to construct a square root $R$ and then find a unitary involution
	$K$ such that $[RK]$ is a product of at most two holomorphic involutions.
	Since $R=(RK)K$, it follows that $[R]$ is a product of at most three
	holomorphic involutions. The identity for the square of a product of three
	involutions then shows that $[R^2]$ is a product of at most four holomorphic
	involutions.
	
	We first define square-root involution length and give the elementary
	algebraic reductions needed for three-involution square roots.
	
	\begin{definition}\label{def:root-involution-length}
		For $g\in\PU(n,1)$, define
		\[
		\rootlinv(g)
		=
		\inf\{\linv(h):h\in\PU(n,1),\ h^2=g\}
		\in\mathbb Z_{\ge0}\cup\{+\infty\},
		\]
		where, as before, $\inf\varnothing=+\infty$. If $g$ has a square
		root of finite involution length, then this infimum is in fact a
		minimum. The \emph{square-root involution length} of the group is
		defined by
		\[
		\rootlinv(\PU(n,1))
		=
		\sup_{g\in\PU(n,1)}\rootlinv(g).
		\]
	\end{definition}
	
	\begin{lemma}\label{lem:three-square-four}
		Suppose that $a^2=b^2=e^2=1$, where some of the factors are allowed
		to be the identity. Then
		\[
		(abe)^2=(aba)(aea)be=[ab,eb]_{\mathrm c}.
		\]
		In particular, if $h^2=g$ and $\linv(h)\le3$, then
		\[
		\linv(g)\le4.
		\]
	\end{lemma}
	
	\begin{proof}
		Since $a^2=1$,
		\[
		(aba)(aea)be
		=abea\,be
		=(abe)^2.
		\]
		Here $aba$ and $aea$ are conjugates of $b$ and $e$, respectively,
		and are therefore involutions or the identity. On the other hand, from
		$a^2=b^2=e^2=1$ we obtain
		\[
		[ab,eb]_{\mathrm c}
		=(ab)(eb)(ba)(be)
		=abeabe.
		\]
		If $\linv(h)\le3$, pad a decomposition with identities to write
		\[
		h=abe,\qquad a,b,e\in\mathcal I\cup\{1\}.
		\]
		The identity above expresses $g=h^2$ as a product of at most four
		nontrivial holomorphic involutions, so $\linv(g)\le4$.
	\end{proof}
	
	\begin{lemma}\label{lem:two-square-closed}
		If $h\in\mathcal I^{\le2}$, then $h^2\in\mathcal I^{\le2}$.
	\end{lemma}
	
	\begin{proof}
		After padding with identities, write $h=ab$, where $a^2=b^2=1$.
		Then
		$h^2=abab=a(bab)$, and both $a$ and $bab$ are involutions or the
		identity. Hence $h^2\in\mathcal I^{\le2}$.
	\end{proof}
	
	\begin{lemma}\label{lem:sqrt-three-reduction}
		Let $G,R\in\UU(n,1)$ and $K\in\UU(n,1)$ satisfy
		\[
		G=R^2,\qquad K^2=\Id.
		\]
		If $[RK]$ is a product of at most two holomorphic involutions, then
		\[
		h=[R],\qquad h^2=[G],\qquad \linv(h)\le3.
		\]
	\end{lemma}
	
	\begin{proof}
		Since $R=(RK)K$, we have $[R]=[RK][K]$. If $[K]\ne1$, then $[K]$
		is a nontrivial holomorphic involution because $K^2=\Id$; if
		$[K]=1$, this factor may simply be omitted. Since $[RK]$ is a product
		of at most two holomorphic involutions, $\linv([R])\le3$. Taking
		$h=[R]$, we have $h^2=[R]^2=[R^2]=[G]$.
	\end{proof}
	
	To construct the required unitary involution on a two-dimensional Hermitian
	space, we first record a reflection formula for vectors of equal norm.
	
	\begin{lemma}\label{lem:equal-norm-exchange}
		Suppose that $v,w$ satisfy
		\[
		\ip{v}{v}=\ip{w}{w},
		\qquad
		\ip{v}{w}\in\R.
		\]
		If $q=v-w$ is positive, then the reflection in the positive line
		\[
		J_q=\Id-2P_{\C q}
		\]
		satisfies $J_qv=w$. If $v=w$, take any unitary involution fixing $v$.
	\end{lemma}
	
	\begin{proof}
		Since $\ip{v}{w}\in\R$, Hermitian symmetry gives
		\[
		\ip{w}{v}=\ip{v}{w}.
		\]
		Using $\ip{v}{v}=\ip{w}{w}$, we obtain
		\[
		\begin{aligned}
			\ip{q}{q}
			&=\ip{v-w}{v-w}\\
			&=2\bigl(\ip{v}{v}-\ip{w}{v}\bigr)
			=2\ip{q}{v}.
		\end{aligned}
		\]
		Therefore
		\[
		P_{\C q}v
		=q\,\frac{\ip{q}{v}}{\ip{q}{q}}
		=\frac12q,
		\]
		and hence
		\[
		J_qv=v-2P_{\C q}v=v-q=w.
		\]
		If $v=w$, it suffices to choose any unitary involution fixing $v$.
	\end{proof}
	
	We next give a phase-selection lemma and a lemma constructing a
	negative-type eigenline on a two-dimensional indefinite block.
	
	\begin{lemma}\label{lem:phase-grid}
		Let $\Delta\in\T$ and $d\in\mathbb Z_{\ge2}$. The $2d$ solutions of
		\[
		c^{2d}=\Delta,\qquad c\in\T,
		\]
		are equally spaced on the unit circle, and adjacent solutions have
		arguments differing by $\pi/d$. Consequently, every open arc of
		length strictly greater than $\pi/d$ contains a solution.
	\end{lemma}
	
	\begin{proof}
		If $\Delta=e^{i\theta}$, all solutions are
		\[
		c_k=\exp\!\left(\frac{i(\theta+2\pi k)}{2d}\right),
		\qquad k=0,\ldots,2d-1.
		\]
		The arguments of adjacent solutions differ by $\pi/d$. If an open
		arc as above contained no solution, it would have to lie in the open
		arc between two adjacent solutions and would therefore have length at
		most $\pi/d$, a contradiction.
	\end{proof}
	
	On a Hermitian plane $H$ of signature $(1,1)$, define, for
	$B\in\UU(H)$,
	\[
	\mathcal W_-(B)
	=
	\bigl\{-\ip{v}{Bv}:\ip{v}{v}=-1\bigr\},
	\qquad
	\Phi(B)
	=
	\left\{\frac z{|z|}:z\in\mathcal W_-(B)\right\}.
	\]
	
	\begin{lemma}\label{lem:negative-activity}
		Let $B\in\UU(1,1)$ and choose
		$z=-\ip{v}{Bv}\in\mathcal W_-(B)$. Set
		\[
		c=\frac z{|z|}.
		\]
		Then there exists a unitary involution $K_H$ of determinant $-1$
		such that $BK_H$ is semisimple and has a negative-type eigenline with
		corresponding eigenvalue $c$. Its other, positive-type eigenline has
		corresponding eigenvalue
		\begin{equation}\label{eq:activity-positive-eigenvalue}
			p=-\frac{\det B}{c}.
		\end{equation}
	\end{lemma}
	
	\begin{proof}
		With respect to the orthogonal decomposition
		$H=\C v\perp v^\perp$, write
		\[
		Bv=zv+u,\qquad u\perp v.
		\]
		Since $\ip{Bv}{Bv}=\ip{v}{v}=-1$,
		\[
		\ip{u}{u}=|z|^2-1,
		\]
		and therefore $|z|\ge1$. Put
		\[
		w=cB^{-1}v.
		\]
		Then $\ip{w}{w}=-1$, while
		\[
		\ip{v}{w}
		=c\,\ip{v}{B^{-1}v}
		=c\,\ip{Bv}{v}
		=-c\overline z=-|z|\in\R.
		\]
		If $|z|>1$, then
		\[
		\ip{v-w}{v-w}=2(|z|-1)>0,
		\]
		so Lemma~\ref{lem:equal-norm-exchange} gives a reflection $K_H$ in
		a positive line such that $K_Hv=w$. If $|z|=1$, the decomposition
		above gives $u=0$ and hence $w=v$; in this case take the reflection in
		the positive line $v^\perp$. In both cases, $\det K_H=-1$, and
		\[
		BK_Hv=Bw=cv.
		\]
		Since $BK_H$ is unitary and preserves the negative line $\C v$, it
		also preserves its orthogonal complement $v^\perp$. The latter is a
		positive line, so $BK_H$ is semisimple. If $p$ denotes its
		positive-type eigenvalue, comparison of determinants gives
		$cp=\det(BK_H)=-\det B$, which is precisely
		Equation~\eqref{eq:activity-positive-eigenvalue}.
	\end{proof}
	
	We can now give the necessary and sufficient spectral pairing criterion on
	the positive-definite orthogonal complement. It characterizes exactly when,
	after choosing signs of eigenvalue square roots and multiplying by a unitary
	involution, the resulting spectrum can be made invariant under a prescribed
	inversion. Condition~\eqref{eq:positive-completion-determinant} is the sole
	condition for such a spectral pairing. We allow $W=\{0\}$; in this case the
	determinant of its identity map is understood to be $1$, and its spectrum is
	the empty multiset.
	
	\begin{proposition}\label{prop:sqrt-positive-completion}
		Let $W>0$, let $\dim_{\C}W=m$, let $D\in\UU(W)$, and let
		$p,c\in\T$. Define
		\[
		\rho_c(\zeta)=\frac{c^2}{\zeta}.
		\]
		Then the following conditions are equivalent:
		\begin{enumerate}[label=\textup{(\roman*)}]
			\item
			\begin{equation}\label{eq:positive-completion-determinant}
				p^2\det D=c^{2(m+1)};
			\end{equation}
			\item there exist $F\in\UU(W)$ and a unitary involution
			$K_W\in\UU(W)$ such that
			\[
			F^2=D,
			\qquad
			\{p\}\sqcup\spec(FK_W)
			\quad\text{is invariant under }\rho_c.
			\]
		\end{enumerate}
	\end{proposition}
	
	\begin{proof}
		We first prove necessity. Put
		\[
		\Pi:=p\det(FK_W),
		\]
		the product of all $m+1$ elements in the spectral multiset
		\[
		\{p\}\sqcup\spec(FK_W).
		\]
		If this multiset is invariant under $\rho_c$, then
		\[
		\Pi
		=\prod_{\zeta}\rho_c(\zeta)
		=\frac{c^{2(m+1)}}{\Pi}.
		\]
		On the other hand, $F^2=D$ and $K_W^2=\Id$ imply
		\[
		\Pi^2
		=p^2\det(FK_W)^2
		=p^2\det D.
		\]
		Thus $p^2\det D=c^{2(m+1)}$, which is
		Equation~\eqref{eq:positive-completion-determinant}.
		
		We prove sufficiency by induction on $m$. If $m=0$, the condition is
		$p^2=c^2$, so $\rho_c(p)=p$. If $m=1$, write $D=a^2$ with
		$a\in\T$. The condition allows us to choose
		$\sigma\in\{\pm1\}$ such that
		\[
		pa\sigma=c^2.
		\]
		Set $F=a$ and $K_W=\sigma\Id$. Then $FK_W=c^2/p$, which is paired
		with $p$.
		
		Now suppose that $m\ge2$. In an orthonormal eigenbasis, write
		\[
		D=\diag(a^2,b^2)\orth D',
		\qquad a,b\in\T,
		\]
		and put $x=c^2/p$. Choose $\varepsilon\in\{\pm1\}$ such that
		\[
		\operatorname{Im}\!\left(\frac xa\right)
		\operatorname{Im}\!\left(\frac{x}{\varepsilon b}\right)\le0.
		\]
		Such a choice is always possible because replacing $b$ by $-b$
		replaces $x/b$ by its negative. In other words, $x/a$ and
		$x/(\varepsilon b)$ lie in opposite closed half-planes bounded by the
		real axis, where points on the real axis are allowed. Some convex
		combination of their imaginary parts therefore vanishes. Thus there
		exists $t\in[0,1]$ such that
		\[
		t\frac xa+(1-t)\frac{x}{\varepsilon b}\in\R.
		\]
		On this two-dimensional block, set
		\[
		F_0=\diag(a,\varepsilon b),\qquad
		v=(\sqrt t,\sqrt{1-t})^T,\qquad
		w=xF_0^{-1}v.
		\]
		Both $v$ and $w$ are unit vectors, and
		\[
		\ip{v}{w}
		=t\frac xa+(1-t)\frac{x}{\varepsilon b}\in\R.
		\]
		By Lemma~\ref{lem:equal-norm-exchange}, there is a
		two-dimensional unitary involution $K_0$ of determinant $-1$ such that
		$K_0v=w$; if $v=w$, take the reflection in $v^\perp$. Hence
		\[
		F_0K_0v=xv.
		\]
		The other eigenvalue of $F_0K_0$ is uniquely determined by its
		determinant and is
		\[
		q=-\frac{\varepsilon ab}{x}.
		\]
		On the remaining $(m-2)$-dimensional space,
		\[
		q^2\det D'
		=\frac{\det D}{x^2}
		=\frac{p^2\det D}{c^4}
		=c^{2(m-1)}.
		\]
		This is exactly the induction hypothesis for the remaining
		$(m-2)$-dimensional space, with $q$ as the prescribed spectral value.
		Thus there exist $F'$ and $K'$ such that
		\[
		(F')^2=D',
		\qquad
		\{q\}\sqcup\spec(F'K')
		\quad\text{is invariant under }\rho_c.
		\]
		Finally, take
		\[
		F=F_0\orth F',\qquad K_W=K_0\orth K'.
		\]
		Now $p$ is paired with $x=\rho_c(p)$, and the remaining spectrum is
		paired under $\rho_c$ by the induction hypothesis. This
		proves~\textup{(ii)}.
	\end{proof}
	
	We conclude this section with the strong-reversibility criterion
	corresponding to the preceding spectral pairing in the loxodromic case.
	
	\begin{lemma}\label{lem:lox-strong-reversible}
		Let $A\in\UU(n,1)$ be semisimple and admit an $A$-invariant
		orthogonal decomposition
		\[
		V=H\orth P,
		\qquad
		\operatorname{sign}(H)=(1,1),\quad P>0.
		\]
		Suppose that the eigenvalues of $A|_H$ on its two null eigenlines are
		\[
		cr,\qquad cr^{-1},
		\qquad c\in\T,\quad r>1,
		\]
		and that $\spec(A|_P)$ is invariant under
		\[
		\rho_c(\zeta)=\frac{c^2}{\zeta}.
		\]
		Then $[A]$ is a product of at most two holomorphic involutions.
	\end{lemma}
	
	\begin{proof}
		Set
		\[
		U=c^{-1}A.
		\]
		Since $[U]=[A]$, it is enough to prove that $[U]$ is a product of at
		most two holomorphic involutions. The two eigenvalues of $U|_H$ are
		$r$ and $r^{-1}$.
		
		We first construct a reversing involution on $H$. Choose vectors
		$e_+$ and $e_-$ on the two null eigenlines of $U|_H$, normalized so
		that
		\[
		\ip{e_+}{e_+}=\ip{e_-}{e_-}=0,
		\qquad
		\ip{e_+}{e_-}=1,
		\]
		and
		\[
		Ue_+=re_+,\qquad Ue_-=r^{-1}e_-.
		\]
		Define
		\[
		J_He_+=-e_-,
		\qquad
		J_He_-=-e_+.
		\]
		Clearly $J_H^2=\Id$. With respect to the basis $(e_+,e_-)$, the
		matrix of $J_H$ and the Gram matrix of the Hermitian form are,
		respectively,
		\[
		\begin{pmatrix}
			0&-1\\
			-1&0
		\end{pmatrix},
		\qquad
		\begin{pmatrix}
			0&1\\
			1&0
		\end{pmatrix}.
		\]
		Thus $J_H$ preserves the Hermitian form and is a unitary involution.
		Moreover, $J_HU|_HJ_H=(U|_H)^{-1}$. Its $(-1)$-eigenline is
		$\C(e_++e_-)$, and
		$\ip{e_++e_-}{e_++e_-}=2>0$, so $J_H$ is a standard involution.
		
		We now treat the positive-definite space $P$. If $\zeta$ is an
		eigenvalue of $A|_P$, then $w=\zeta/c$ is an eigenvalue of $U|_P$,
		and
		\[
		\frac{\rho_c(\zeta)}{c}
		=\frac{c}{\zeta}
		=w^{-1}.
		\]
		Hence $\spec(U|_P)$ is invariant under $w\mapsto w^{-1}$. Since
		$P>0$ and $U|_P$ is unitary, it is unitarily diagonalizable. Thus,
		for every $w\ne\pm1$,
		\[
		E_w=\ker(U-w\Id),
		\qquad
		E_{w^{-1}}=\ker(U-w^{-1}\Id)
		\]
		have the same dimension and are orthogonal.
		
		For every pair $E_w,E_{w^{-1}}$, choose a unitary isometry
		\[
		\phi_w:E_w\longrightarrow E_{w^{-1}},
		\]
		and define, on $E_w\orth E_{w^{-1}}$,
		\[
		J_P(x+y)
		=\phi_w(x)+\phi_w^{-1}(y),
		\qquad
		x\in E_w,\ y\in E_{w^{-1}}.
		\]
		Set $J_P=\Id$ on the eigenspaces for the eigenvalues $\pm1$.
		Taking the orthogonal direct sum of these maps gives a unitary
		involution $J_P$ satisfying
		$J_PU|_PJ_P=(U|_P)^{-1}$.
		
		Now set
		\[
		J=J_H\orth J_P.
		\]
		Then $J$ is a linear unitary involution and $JUJ=U^{-1}$. If
		\[
		K=JU,
		\]
		then
		$K^2=JUJU=U^{-1}U=\Id$, so $K$ is also a linear unitary involution,
		and $U=JK$.
		
		Since $J^2=K^2=\Id$, Lemma~\ref{lem:normalise-involution} shows that
		each of $[J]$ and $[K]$, when nontrivial, can be represented by a
		standard linear involution; if either projective class is the identity,
		it is simply omitted. Thus $[A]=[U]=[J][K]$ is a product of at most
		two holomorphic involutions.
	\end{proof}
	
	\section{Square-Root Constructions for the Different Isometry Types}
	\label{sec:sqrt-types}
	
	\subsection{Square-root constructions on two-dimensional active blocks}
	\label{sec:upper}
	
	Elliptic, loxodromic, and two-step unipotent isometries share the same
	construction: one first produces a negative-type eigenline on the
	two-dimensional active block and then completes the square-root decomposition
	using spectral pairing on the positive-definite orthogonal complement.
	
	\begin{lemma}\label{lem:active-root-closure}
		Let $g=[G]\in\PU(n,1)$, and suppose that there is an orthogonal
		decomposition
		\[
		V=H\orth W,
		\qquad
		\operatorname{sign}(H)=(1,1),\quad W>0,\quad
		d=\dim_{\C}V,
		\]
		where
		\[
		G=B^2\orth D,
		\qquad B\in\UU(H),\quad D\in\UU(W).
		\]
		If there exists $c\in\Phi(B)$ satisfying
		\begin{equation}\label{eq:global-phase-equation}
			c^{2d}=\det G,
		\end{equation}
		then $g$ has a square root $h$ such that $\linv(h)\le3$.
	\end{lemma}
	
	\begin{proof}
		By Lemma~\ref{lem:negative-activity}, there exists a unitary
		involution $K_H$ such that $BK_H$ is semisimple, its negative-type
		eigenline has corresponding eigenvalue $c$, and its positive-type
		eigenline has corresponding eigenvalue
		\[
		p=-\frac{\det B}{c}.
		\]
		Let $m=\dim_{\C}W=d-2$. Since
		$\det G=(\det B)^2\det D$,
		Equation~\eqref{eq:global-phase-equation} gives
		\[
		p^2\det D
		=\frac{(\det B)^2\det D}{c^2}
		=c^{2(d-1)}
		=c^{2(m+1)}.
		\]
		Proposition~\ref{prop:sqrt-positive-completion} therefore gives
		$F^2=D$ and $K_W^2=\Id$ such that
		\[
		\{p\}\sqcup\spec(FK_W)
		\]
		is invariant under $\rho_c$.
		
		Set
		\[
		R=B\orth F,\qquad K=K_H\orth K_W.
		\]
		Then $R^2=G$ and $K^2=\Id$. The operator
		\[
		RK=(BK_H)\orth(FK_W)
		\]
		is semisimple; its negative-type eigenline has corresponding
		eigenvalue $c$, and the spectrum on its positive-definite part is
		invariant under $\rho_c$. Lemma~\ref{lem:spectral-strong} shows that
		$[RK]$ is a product of at most two holomorphic involutions. The result
		now follows from Lemma~\ref{lem:sqrt-three-reduction}.
	\end{proof}
	
	\begin{proposition}\label{prop:sqrt-two-dimensional}
		Let $n\ge2$, and let $g\in\PU(n,1)$ be elliptic, loxodromic, or
		two-step unipotent. Then $g$ has a square root $h$ satisfying
		$\linv(h)\le3$.
	\end{proposition}
	
	\begin{proof}
		\medskip\noindent\textit{(i) The elliptic case.}
		The assertion is immediate for the identity, so suppose that $g\ne1$.
		After multiplying a lift by a central scalar, we may arrange that its
		eigenvalue on the negative-type eigenline is $1$. Choose a
		positive-type eigenline with eigenvalue $\delta\ne1$, and write
		\[
		\delta=e^{2i\beta},\qquad 0<\beta<\pi.
		\]
		On the two-dimensional indefinite subspace spanned by the chosen
		negative- and positive-type eigenlines, take the corresponding
		orthonormal eigenbasis. Then
		\[
		G=\diag(1,e^{2i\beta})\orth D.
		\]
		Consider the two square roots
		\[
		B_+=\diag(1,e^{i\beta}),\qquad
		B_-=\diag(1,-e^{i\beta}).
		\]
		If $v=(\sqrt{1+s},\sqrt s)^T$, then $\ip{v}{v}=-1$, and direct
		calculation gives
		\[
		\begin{aligned}
			\mathcal W_-(B_+)
			&=\{1+(1-e^{i\beta})s:s\ge0\},\\
			\mathcal W_-(B_-)
			&=\{1+(1+e^{i\beta})s:s\ge0\}.
		\end{aligned}
		\]
		Consequently,
		\[
		\begin{aligned}
			\Phi(B_+)
			&=\{e^{i\gamma}:(\beta-\pi)/2<\gamma\le0\},\\
			\Phi(B_-)
			&=\{e^{i\gamma}:0\le\gamma<\beta/2\}.
		\end{aligned}
		\]
		Their union is an open arc of length $\pi/2$.
		
		Let $d=n+1\ge3$. Adjacent solutions of
		$c^{2d}=\det G$ on the unit circle have arguments differing by
		$\pi/d\le\pi/3<\pi/2$. Lemma~\ref{lem:phase-grid} therefore allows
		us to choose a solution $c$ in the union above. According as
		$c\in\Phi(B_+)$ or $c\in\Phi(B_-)$, choose the corresponding square
		root. Since $B_\pm^2=\diag(1,e^{2i\beta})$,
		Lemma~\ref{lem:active-root-closure} gives the required square root.
		
		\medskip\noindent\textit{(ii) The loxodromic case.}
		By Lemma~\ref{lem:active-blocks}, in a normalized null basis we may
		write
		\[
		G=\mu^2\diag(r^2,r^{-2})\orth D,
		\qquad \mu\in\T,\quad r>1,\quad
		D\in\UU(W),\quad W>0.
		\]
		Take
		\[
		B=\mu\diag(r,r^{-1}).
		\]
		If the Gram matrix of the null basis is
		$\left(\begin{smallmatrix}0&1\\1&0\end{smallmatrix}\right)$, then a
		negative unit vector satisfies
		$\overline v_1v_2=-1/2+it$. Hence
		\[
		\mathcal W_-(B)
		=
		\mu\left\{
		\frac{r+r^{-1}}2+is:s\in\R
		\right\},
		\]
		where a nonzero real factor has been absorbed into the parameter $s$.
		It follows that
		\[
		\Phi(B)
		=\{\mu e^{i\gamma}:|\gamma|<\pi/2\},
		\]
		an open arc of length $\pi$.
		
		Let $d=n+1\ge3$. Adjacent solutions of $c^{2d}=\det G$ on the unit
		circle have arguments differing by $\pi/d$, so
		Lemma~\ref{lem:phase-grid} gives a $c\in\Phi(B)$. Since
		$B^2=\mu^2\diag(r^2,r^{-2})$, the assertion follows from
		Lemma~\ref{lem:active-root-closure}.
		
		\medskip\noindent\textit{(iii) The two-step unipotent case.}
		In the null basis from Lemma~\ref{lem:active-blocks}, write
		\[
		G=
		\lambda
		\begin{pmatrix}1&it\\0&1\end{pmatrix}
		\orth D,
		\qquad t\in\R\setminus\{0\}.
		\]
		Choose $\mu^2=\lambda$, put $u=t/2$, and set
		\[
		B=
		\mu
		\begin{pmatrix}1&iu\\0&1\end{pmatrix}.
		\]
		Then $B^2$ is precisely the active block above. If $v$ is a negative
		unit vector, set $s=|v_2|^2>0$ in the same null basis. Then
		\[
		-\ip{v}{Bv}=\mu(1-ius).
		\]
		Thus
		\[
		\Phi(B)=
		\begin{cases}
			\{\mu e^{i\gamma}:-\pi/2<\gamma<0\},&u>0,\\
			\{\mu e^{i\gamma}:0<\gamma<\pi/2\},&u<0.
		\end{cases}
		\]
		This is an open arc of length $\pi/2$. Since $d=n+1\ge3$,
		$\pi/d\le\pi/3<\pi/2$, and Lemma~\ref{lem:phase-grid} allows us to
		choose $c\in\Phi(B)$ satisfying $c^{2d}=\det G$.
		Lemma~\ref{lem:active-root-closure} completes the proof.
	\end{proof}
	
	\subsection{The square-root construction for three-step unipotent isometries}
	\label{sec:three-step}
	
	In the two-dimensional case, spectral pairing is accomplished through one
	negative-type eigenline. The active block of a three-step unipotent isometry
	is three-dimensional. We first take a square root of this pure unipotent
	block and put it into standard form, then construct a loxodromic
	two-dimensional block, and finally combine it with the positive-definite
	orthogonal complement to complete the construction.
	
	\begin{proposition}[Square roots of three-step unipotent isometries]
		\label{prop:sqrt-three-step}
		Let $n\ge2$. Every three-step unipotent element $g\in\PU(n,1)$ has a
		square root $h$ satisfying $\linv(h)\le3$.
	\end{proposition}
	
	\begin{proof}
		By Lemma~\ref{lem:active-blocks}, one may choose a lift and an
		orthogonal decomposition
		\[
		V=H_3\orth W,
		\qquad
		\operatorname{sign}(H_3)=(2,1),\quad W>0,
		\]
		such that
		\[
		G=U\orth D,
		\qquad
		U\in\UU(H_3),
		\qquad
		(U-\Id)^3=0,\quad (U-\Id)^2\ne0,
		\qquad D\in\UU(W).
		\]
		Here $U$ is a pure three-step unipotent operator on the
		three-dimensional active subspace $H_3$. Since $U-\Id$ is nilpotent,
		set
		\[
		Y=\log U
		=(U-\Id)-\frac12(U-\Id)^2.
		\]
		The identity $U^\dagger=U^{-1}$ gives $Y^\dagger=-Y$. Since
		$Y^3=0$ and $Y^2\ne0$, let
		\[
		M=\exp\!\left(\frac12Y\right).
		\]
		Then $M$ is unitary and $M^2=\exp Y=U$. Moreover,
		\[
		M-\Id=\frac12Y+\frac18Y^2,
		\]
		so $(M-\Id)^3=0$ and $(M-\Id)^2\ne0$. Thus $M$ is again a pure
		three-step unipotent operator. By the standard form for three-step
		unipotent parabolic isometries
		\cite[Section~3.3.2]{PaupertWill2017}, after unitary conjugation we may
		assume that
		\[
		M=N=
		\begin{pmatrix}
			1&-\sqrt2&-1\\
			0&1&\sqrt2\\
			0&0&1
		\end{pmatrix}.
		\]
		In these standard coordinates, the Hermitian form is represented by
		\[
		\mathsf H_3=
		\begin{pmatrix}
			0&0&1\\
			0&1&0\\
			1&0&0
		\end{pmatrix}.
		\]
		Therefore
		\begin{equation}\label{eq:three-step-square-lift}
			G=N^2\orth D.
		\end{equation}
		
		Let $d=\dim_{\C}V=n+1$ and $m=\dim_{\C}W=d-3$. Take
		$\theta=\operatorname{Arg}(\det G)\in(-\pi,\pi]$ and set
		\begin{equation}\label{eq:three-step-principal-phase}
			c=\exp\!\left(\frac{i\theta}{2d}\right),
			\qquad p=-c^{-2}.
		\end{equation}
		Then
		\[
		c^{2d}=\det G=\det D,
		\qquad
		|\operatorname{Arg}c|\le\frac{\pi}{2d}\le\frac{\pi}{6}.
		\]
		
		Let
		\[
		u=(1/2,0,1)^T,
		\qquad
		v=Nu=(-1/2,\sqrt2,1)^T.
		\]
		A direct calculation gives
		\[
		\ip{u}{u}=\ip{v}{v}=1,
		\qquad
		\ip{u}{v}=0,
		\qquad
		\ip{u}{N^2u}=-3.
		\]
		Put $q=v-pu$. Then $\ip{q}{q}=2$. Since $|p|=1$, the vectors $v$
		and $pu$ have the same norm, and $\ip{v}{pu}=0\in\R$. Hence,
		by Lemma~\ref{lem:equal-norm-exchange}, the reflection in the positive
		line $\C q$,
		\[
		J_H=\Id-2P_{\C q},
		\]
		satisfies
		\[
		J_Hv=pu,\qquad NJ_Hv=pv.
		\]
		Thus $\C v$ is a positive-type eigenline of $NJ_H$ with
		corresponding eigenvalue $p$.
		
		Using $v=Nu$, $Nq=Nv-pv$, unitarity, and the three inner products
		above, we obtain
		\[
		\begin{aligned}
			\ip{q}{Nq}
			&=\ip{v-pu}{Nv-pv}\\
			&=\ip{v}{Nv}-p\ip{v}{v}
			-\overline p\ip{u}{Nv}+\overline pp\ip{u}{v}\\
			&=\ip{u}{Nu}-p-\overline p\ip{u}{N^2u}
			=-p+3p^{-1}.
		\end{aligned}
		\]
		With our convention for the Hermitian form, the rank-one trace formula
		gives
		\[
		P_{\C q}x
		=q\frac{\ip{q}{x}}{\ip{q}{q}},
		\qquad
		\tr(NP_{\C q})
		=\frac{\ip{q}{Nq}}{\ip{q}{q}}.
		\]
		The second identity corresponds to the rank-one operator
		$x\mapsto Nq\,\ip{q}{x}/\ip{q}{q}$. Since $\tr N=3$, it follows that
		\[
		\tr(NJ_H)
		=\tr N-2\tr(NP_{\C q})
		=3-2\frac{\ip{q}{Nq}}{\ip{q}{q}}
		=3+p-3p^{-1}.
		\]
		
		The operator $NJ_H$ is unitary and preserves the nonisotropic
		eigenline $\C v$, so it also preserves its orthogonal complement in
		$H_3$. Let
		\[
		H_v:=H_3\cap v^\perp,
		\qquad
		L=(NJ_H)|_{H_v}.
		\]
		The space $H_v$ has signature $(1,1)$. Since $\det N=1$ and
		$\det J_H=-1$, computing the determinant along the orthogonal
		decomposition $H_3=\C v\orth H_v$ gives
		\[
		p\det L=\det(NJ_H)=-1.
		\]
		Subtracting the eigenvalue $p$ from the total trace and using
		$p=-c^{-2}$, we obtain
		\[
		\det L=-p^{-1}=c^2,
		\qquad
		\tr L=\tr(NJ_H)-p
		=3-3p^{-1}=3(1+c^2).
		\]
		If $c=e^{i\gamma}$, then $S=c^{-1}L\in\UU(1,1)$, and
		\[
		\det S=1,
		\qquad
		\tr S=3(c+c^{-1})=6\cos\gamma>2;
		\]
		the strict inequality follows from $|\gamma|\le\pi/6$. Hence the
		characteristic polynomial of $S$ is
		\[
		t^2-(\tr S)t+1,
		\]
		and has two distinct positive real roots $r,r^{-1}$, with $r>1$. If
		$Sx=rx$, then
		\[
		\ip{x}{x}=\ip{Sx}{Sx}=r^2\ip{x}{x},
		\]
		so $\ip{x}{x}=0$; the same argument applies to the root $r^{-1}$.
		The roots are distinct, so $S$ is semisimple and both its eigenlines
		are null. Multiplying back by the phase $c$, we see that $L$ is
		semisimple and its eigenvalues on the two null eigenlines are
		\[
		cr,\qquad cr^{-1}
		\qquad(r>1).
		\]
		
		It remains to pair the spectrum on the positive-definite orthogonal
		complement. By~\eqref{eq:three-step-principal-phase},
		\[
		p^2\det D
		=c^{-4}c^{2d}
		=c^{2(d-2)}
		=c^{2(m+1)}.
		\]
		Proposition~\ref{prop:sqrt-positive-completion} gives $F^2=D$ and
		$K_W^2=\Id$ such that
		\[
		\{p\}\sqcup\spec(FK_W)
		\]
		is invariant under $\rho_c$. Finally, set
		\[
		R=N\orth F,\qquad K=J_H\orth K_W.
		\]
		By~\eqref{eq:three-step-square-lift}, $R^2=G$. On the
		two-dimensional space $H_v$ of signature $(1,1)$, the spectrum of
		$RK$ is $cr,cr^{-1}$, while its spectrum on the positive-definite
		space $\C v\orth W$ is invariant under $\rho_c$. Therefore
		Lemma~\ref{lem:lox-strong-reversible} shows that $[RK]$ is a product
		of at most two holomorphic involutions. The result follows from
		Lemma~\ref{lem:sqrt-three-reduction}.
	\end{proof}
	
	\begin{corollary}\label{cor:sqrt-upper}
		For every $n\ge2$ and every $g\in\PU(n,1)$, there exists
		$h\in\PU(n,1)$ such that
		\[
		h^2=g,\qquad \linv(h)\le3.
		\]
	\end{corollary}
	
	\begin{proof}
		If $g=1$, take $h=1$. If $g\ne1$, Lemma~\ref{lem:active-blocks}
		gives four cases. Apply Proposition~\ref{prop:sqrt-two-dimensional}
		in the first three cases and Proposition~\ref{prop:sqrt-three-step} in
		the three-step unipotent case.
	\end{proof}
	
	\section{Proofs of the Main Theorems and the Commutator Representation}
	\label{sec:sqrt-conclusion}
	
	\begin{proof}[Proof of Theorem~\ref{thm:sqrt-main}]
		For every $n\ge2$, Corollary~\ref{cor:sqrt-upper} gives
		\[
		\rootlinv(\PU(n,1))\le3.
		\]
		On the other hand, choose $\lambda\in\T\setminus\{\pm1\}$. By
		Lemma~\ref{lem:point-not-two},
		$[A_\lambda]\notin\mathcal I^{\le2}$. If there were a square root
		$h$ satisfying $h^2=[A_\lambda]$ and $\linv(h)\le2$, then
		Lemma~\ref{lem:two-square-closed} would imply
		$[A_\lambda]=h^2\in\mathcal I^{\le2}$, a contradiction. Hence every
		square root of $[A_\lambda]$ has involution length at least $3$.
		Combining this with the upper bound gives
		\[
		\rootlinv([A_\lambda])=3,
		\qquad
		\rootlinv(\PU(n,1))=3.
		\]
	\end{proof}
	
	\begin{proof}[Proof of Theorem~\ref{thm:main}]
		Let $g\in\PU(n,1)$. By Corollary~\ref{cor:sqrt-upper}, there exists
		$h\in\PU(n,1)$ such that $h^2=g$ and $\linv(h)\le3$. Write $h$ as a
		product of at most three involutions and pad with identity factors so that
		$h=abe$. Lemma~\ref{lem:three-square-four} gives
		\[
		g=(abe)^2=(aba)(aea)be,
		\]
		so $\linv(g)\le4$. Hence $\linv(\PU(n,1))\le4$. Conversely, for
		every $n\ge3$, Corollary~\ref{cor:lower} provides a point rotation
		that cannot be written as a product of at most three
		holomorphic involutions. Thus $\linv(\PU(n,1))\ge4$, and the two
		inequalities prove the theorem.
	\end{proof}
	
	Combining the two theorems above with the two-dimensional result of
	Paupert--Will, the exact values can be written uniformly as
	\[
	\linv(\PU(n,1))=4,
	\qquad
	\rootlinv(\PU(n,1))=3
	\qquad(n\ge2).
	\]
	
	\begin{proof}[Proof of Corollary~\ref{cor:commutator-main}]
		Let $g\in\PU(n,1)$. By Corollary~\ref{cor:sqrt-upper}, there exists
		$h\in\PU(n,1)$ such that
		\[
		h^2=g,\qquad \linv(h)\le3.
		\]
		We may write
		\[
		h=abe,\qquad a^2=b^2=e^2=1,
		\]
		with $b$ a nontrivial holomorphic involution. Indeed, if a shortest
		decomposition has three factors, take its middle factor as $b$. If
		$\linv(h)=2$, write $h=a_0b_0$ and take
		\[
		(a,b,e)=(a_0,b_0,1).
		\]
		If $\linv(h)=1$, choose any nontrivial holomorphic involution $\iota$
		and write $h=h\iota\iota$. If $h=1$, write $h=\iota\iota1$.
		
		Set
		\[
		x=ab,\qquad y=eb,\qquad I=b.
		\]
		By Lemma~\ref{lem:three-square-four},
		$g=h^2=[x,y]_{\mathrm c}$. Moreover,
		\[
		IxI=b(ab)b=ba=x^{-1},
		\]
		and
		\[
		IyI=b(eb)b=be=y^{-1}.
		\]
		Thus $I$ simultaneously reverses $x$ and $y$.
		
		Therefore every element of $\PU(n,1)$ is a single commutator. Since
		$\PU(n,1)$ is nontrivial, $\operatorname{cl}(\PU(n,1))=1$.
	\end{proof}

\end{document}